\documentclass{amsart}

\usepackage{ucs}

\usepackage{fancybox}
\usepackage{color}
\usepackage{latexsym}

\usepackage{eepic}
\usepackage{epic}
\usepackage{epsf}
\usepackage{graphics}

\usepackage{amssymb}
\usepackage{amsthm}
\usepackage{amsmath}
\usepackage{latexsym}
\usepackage[cp1251]{inputenc}
\usepackage{graphicx}
\usepackage{wrapfig}
\usepackage{caption}
\usepackage{subcaption}
\usepackage{indentfirst}
\usepackage{enumerate}
\usepackage{makecell}

\def\cF{{\mathcal F}}

\DeclareMathOperator{\aut}{Aut}

\DeclareMathOperator{\cay}{Cay}

\DeclareMathOperator{\id}{id}

\DeclareMathOperator{\iso}{Iso}

\DeclareMathOperator{\sym}{Sym}

\DeclareMathOperator{\poly}{poly}

\DeclareMathOperator{\WL}{WL}

\makeatletter 
\def\@seccntformat#1{\csname the#1\endcsname. } 
\def\@biblabel#1{#1.}

\makeatother

\title[Isomorphism testing of $k$-spanning tournaments]{Isomorphism
  testing of $k$-spanning tournaments is Fixed Parameter Tractable}

\author{Vikraman Arvind}
\address{The Institute of Mathematical Sciences (HBNI), Chennai, India}
\email{arvind@imsc.res.in}
\author{Ilia Ponomarenko}
\address{St. Petersburg Department of V.A. Steklov Institute of Mathematics, St. Petersburg, Russia}
\email{inp@pdmi.ras.ru}
\author{Grigory Ryabov}
\address{Ben Gurion University of the Negev, Beer Sheva, Israel}
\email{gric2ryabov@gmail.com}

\thanks{The third author is supported by The Israel Science Foundation (project No.~87792731)}

\date{}

\newtheorem{state}{Statement}[section]

\newtheorem{lemm}[state]{Lemma}
\newtheorem{theo}[state]{Theorem}
\newtheorem{prop}[state]{Proposition}

\theoremstyle{definition}

\begin{document}

\vspace{\baselineskip}
\vspace{\baselineskip}

\vspace{\baselineskip}

\vspace{\baselineskip}

\begin{abstract}
An arc-colored tournament is said to be \emph{$k$-spanning} for an
integer $k\geq 1$ if the union of its arc-color classes of maximal
valency at most~$k$ is the arc set of a strongly connected digraph. It is
proved that isomorphism testing of $k$-spanning tournaments is
fixed-parameter tractable.  
\end{abstract}

\maketitle

\section{Introduction}

Over the last two decades, there has been an increased interest in algorithms testing
graph isomorphism that are of fixed parameter tractable complexity (FPT algorithms). 
More precisely, for the
graph isomorphism instances we bound some natural graph parameter by a
value $k$, which defines a graph class for the input instances. An FPT
isomorphism testing algorithm for input graphs from such a class
is required to have running time bounded by $f(k)\cdot n^{O(1)}$,
where $f(\cdot)$ is allowed to be an arbitrary function of the
parameter $k$, but is independent of the input size, and $n$ is the
number of vertices in the input graph.

Some well-studied parameters in this context are, for instance, the
eigenvalue multiplicity of the input graphs~\cite{EP99}, color class
size of colored input hypergraphs~\cite{ADKT}, the treewidth of the
input graphs~\cite{LPPS} etc. However, for the  well known (in the context)
parameter $k$ bounding
the valency of the input graphs, testing isomorphism in FPT remains an open problem
despite some
recent interesting progress \cite{GNS2020}. A remark in the seminal paper
of Babai and Luks~\cite[p.~8]{BL}  could be
useful for constructing such an algorithm, but the paper in preparation
referred there was probably not published.  Motivated by this open
problem, in the present note we describe an FPT algorithm for testing
isomorphism of colored tournaments with respect to a new parameter
that plays an analogous role for arc-colored tournaments as valency
for graphs.

\textbf{The parameter for tournaments.}
Recall that a \emph{tournament} is a digraph obtained by assigning a
direction for each edge in an undirected complete graph without
loops. The automorphism group of a tournament is of odd order and
hence solvable by the Feit-Thompson theorem. Based on this fact, it
was proved in~\cite{BL} that the isomorphism of tournaments can be
tested in quasipolynomial time (of course, this also follows from the
recent result~\cite{Babai}). On the other hand, to date, no algorithms
are known for isomorphism testing of tournaments, which are in FPT for
some parameter. In the present note, we construct such an algorithm
for \emph{$k$-spanning} tournaments defined as follows. An arc-colored
tournament is said to be $k$-spanning for an integer $k\geq 1$ if the
union of its arc-color classes of maximal valency at most~$k$ is the
arc set of a strongly connected digraph (examples of $k$-spanning
tournaments are discussed in Subsection~2.2).

\begin{theo}\label{main}
Let $k\geq 1$. Given $k$-spanning colored tournaments~$X$ and~$Y$, the
set $\iso(X,Y)$ of all isomorphisms from $X$ to~$Y$ can be found in time~$k^{O(\log k)}\poly(n)$, where
$n$ is the number of vertices in $X$.
\end{theo}

It should be remarked that Theorem~\ref{main} can be extended
to a wider class of colored tournaments for which the digraph formed
by the arc-color classes of bounded valency is weakly
connected. Besides, for~$k$ close to~$n$, our algorithm has the same
complexity as the algorithm in~\cite{BL}. In fact, we use the latter
algorithm for $n\le k$ (as well as the algorithm from~\cite{Miller} for
finding the intersection of a solvable group with the automorphism group 
of a hypergraph).

The text of the paper is organized as follows. Section~$2$ contains
the necessary background about colored digraphs, including the
definition and examples of $k$-spanning colored tournaments, wreath
product of permutation groups, and some known algorithms. In
Section~3, we construct an auxiliary procedure which extends an action
of a permutation group from a vertex-color class to the other one or
refine the coloring. In Section~$4$, we prove Theorem~\ref{main}.

\section{Preliminaries}

\subsection{Colored digraphs}

A digraph~$X$ is said to be \emph{colored} if it is equipped with linear
ordered partitions~$\pi$ and~$S$ of the vertex and arc sets of~$X$,
respectively; the elements of $\pi$ and $S$ are referred to
\emph{vertex-color classes} and \emph{arc-color classes}. The index
numbers of vertex- and arc-color classes in the corresponding linear orderings are 
called vertex  and arc \emph{colors}, respectively.

An isomorphism from a colored digraph $X$ to a colored
digraph~$X^{\prime}$ is an ordinary digraph isomorphism that preserves
vertex- and arc-color classes, respecting the corresponding linear
orderings. The coset of all isomorphisms from $X$ to $X^{\prime}$ is
denoted by $\iso(X,X^{\prime})$. The group $\aut(X)=\iso(X,X)$ is
called the \emph{automorphism group} of $X$. If $X$ and $X'$ are
isomorphic, then $\iso(X,X')$ is the right coset
$\aut(X)\pi$, where $\pi\in\iso(X,X')$; of course,
$\iso(X,X^{\prime})$ is empty if $X$ and $X^\prime$ are not
isomorphic.

Suppose $X$ and $X^{\prime}$ have the same vertex set. We write
$X^{\prime}\succeq X$ if each vertex-color class (respectively,
arc-color class) of~$X$ is a union of some vertex-color classes
(respectively, arc-color classes) of~$X^{\prime}$, and
$\aut(X)=\aut(X^{\prime})$. As usual, $X^{\prime}\succ X$ if
$X^{\prime}\succeq X$ and $X^{\prime}\neq X$.

\subsection{$k$-spanning digraphs}\label{140122a}

Let $\Omega$ be the vertex set of the colored digraph $X$ and $k\geq
1$. Denote by $s_k$ the union of all $s\in S$ such that
\begin{equation}\label{valency} 
|\alpha s|\leq k, \qquad \alpha\in \Omega,
\end{equation}
where $\alpha s=\{\beta\in \Omega:~(\alpha,\beta)\in s\}$. We say that
$X$ is \emph{$k$-spanning} if the digraph with vertex set $\Omega$ and arc set~$s_k$ is strongly
connected, i.e., there exists a directed path from every vertex to
any other. Two natural examples of $k$-spanning tournaments are given
below. In what follows, for a binary relation~$s$ the minimal
number~$k$ satisfying  Eq.~\eqref{valency} is called the
\emph{maximal valency} of $s$.\medskip

{\bf Cayley tournaments.}  Let $G$ be a finite group of odd order and
$A$ be an identity-free subset of~$G$ such that~$|A\cap
\{g,g^{-1}\}|=1$ for every nonidentity $g\in G$. Then the Cayley digraph
$X=\cay(G,A)$ is a tournament. Every partition~$A=A_1\cup \ldots \cup
A_m$ induces an ordered partition of the arcs of~$X$, the $i$th class
of which is the arc set of the Cayley digraph $\cay(G,A_i)$,
$i=1,\ldots,m$. The resulting colored tournament (with arbitrarily
chosen vertex-colored classes) is $k$-spanning if and only if the
union of all~$A_i$ of cardinality at most~$k$ is a generating set
of~$G$.  In the special case $|A_i|\le k$ for all~$i$, this always
gives a $k$-spanning (Cayley) tournament, and even for small $k$, the
number of such tournaments is exponential in $|G|$. Note that
recognizing Cayley tournaments efficiently is an open problem, and
polynomial-time algorithms are known only for some groups $G$
close to cyclic, see~\cite{P1992,R2021}.\medskip

{\bf Tournaments with bounded immersion.} Recall that a tournament~$Y$
is immersed in a tournament~$X$ if the vertices of~$Y$ are mapped to
distinct vertices of~$X$ and the arcs of~$Y$ are mapped to directed
paths joining the corresponding pairs of vertices of $X$, in such a
way that these paths are pairwise arc-disjoint (see,
e.g.,~\cite{L2019}). Given a tournament~$X$, denote by $k(X)$ the
maximal integer~$k$ such that a transitive tournament\footnote{A
  tournament is said to be \emph{transitive} if the binary relation
  defined by the arcs is transitive.} with $k$ vertices is immersed
in~$X$. We claim that if $X$ is a tournament with $k(X)\leq k$ and
colored by the Weisfeiler-Leman algorithm (see Subsection~2.4), then
$X$ is $O(k^3)$-spanning. Indeed, for every vertex-color class
$\Delta$ of~$X$, the induced subdigraph $X_{\Delta}$ is a regular
tournament with $k(X_{\Delta})\leq k(X)\leq k$. By~\cite{L2019}, this
yields $|\Delta|=O(k^3)$. Thus, the maximum valency of each arc-color
class of $X$ is at most~$O(k^3)$. It should be noted that every
$n$-vertex tournament has a transitive subtournament with at least
$\log n$ vertices; in particular, $k(X)\ge \log n$. Hence, this
example does not really give a fixed parameter problem. However, the
algorithm of Theorem~\ref{main} is significantly faster than the
$n^{\log n}$ time algorithm for general tournaments, like for example
when $k(X)=O(\log n)^{O(1)}$.


\subsection{Wreath product}

Let $X$ be a digraph with vertex set $\Gamma\sqcup \Delta$ and  the
 intersection $D$ of the arc set of~$X$ with
$\Gamma\times\Delta$ is nonempty. Suppose that for every
$\gamma,\gamma^{\prime}\in \Gamma$, we are given a nonempty set
$H_{\gamma,\gamma^{\prime}}$ of bijections $\gamma D\to \gamma'D$ 
(see notation in Subsection~\ref{140122a}) such that
\begin{equation}\label{property}
H_{\gamma,\gamma^{\prime}}H_{\gamma^{\prime},\gamma^{\prime\prime}}=H_{\gamma,\gamma^{\prime\prime}},\qquad\gamma,\gamma^{\prime},\gamma^{\prime\prime}\in \Gamma,
\end{equation}
where
$H_{\gamma,\gamma^{\prime}}H_{\gamma^{\prime},\gamma^{\prime\prime}}$
is the set of all compositions $\psi\circ \rho$ with
$\psi\in H_{\gamma,\gamma^{\prime}}$ and 
$\rho \in H_{\gamma^{\prime},\gamma^{\prime\prime}}$. Note that the assumption implies  $|\gamma D|=|\gamma'D|$ for all $\gamma,\gamma'$.

For $g\in \sym(\Gamma)$, denote by $\mathcal{F}(g)$ the set of all
full systems of distinct representatives of the
family~$\{H_{\gamma,\gamma^g}:~\gamma\in \Gamma\}$; we have $\cF(g)\ne\varnothing$, because $H_{\gamma,\gamma'}\ne\varnothing$ for all $\gamma,\gamma'$. Given such a
system $F=\{f_{\gamma,\gamma^g}\in H_{\gamma,\gamma^g}:~\gamma\in \Gamma\}$, we define a
permutation~$f_{F}\in \sym(D)$ as follows:
$$
(\gamma,\delta)^{f_{F}}=(\gamma^g,\delta^{f_{\gamma,\gamma^g}}),\qquad \gamma\in \Gamma,\  \delta\in \gamma D.
$$

\begin{lemm}\label{wr}
Let $X$ be the above digraph and $K\leq \sym(\Gamma)$. Then the set
\begin{equation}\label{wreath}
W=W(X,K)=\{f_{F}:~F\in \mathcal{F}(g),~g\in K\}
\end{equation}
is a permutation group on~$D$. Moreover, this group is permutationally
isomorphic to $H_{\gamma,\gamma}\wr K$ for every $\gamma\in \Gamma$.
\end{lemm}

\begin{proof}
Let $\gamma\in \Gamma$. For every $\gamma^{\prime}\in \Gamma$, we fix
an arbitrary $h_{\gamma,\gamma^{\prime}}\in
H_{\gamma,\gamma^{\prime}}$. Then the mapping
$$\varphi:\Gamma\times \gamma D\rightarrow D,~(\gamma^{\prime},\delta)\mapsto(\gamma^{\prime},\delta^{h_{\gamma,\gamma^{\prime}}})$$
is obviously a bijection. It suffices to show that $\varphi W\varphi^{-1}=H_{\gamma,\gamma}\wr K$. To this end, let $f=f_{F}\in W$, where $F=\{f_{\gamma^{\prime},(\gamma^{\prime})^g}:~\gamma^{\prime}\in \Gamma\}\in \mathcal{F}(g)$ and $g\in K$. Then for any $(\gamma^{\prime},\delta)\in \Gamma \times \gamma D$, we have 
$$
(\gamma^{\prime},\delta)^{\varphi f\varphi^{-1}}=
(\gamma^{\prime},\delta^{h_{\gamma,\gamma^{\prime}}})^{f\varphi^{-1}}=((\gamma^{\prime})^g, \delta^{h_{\gamma,\gamma^{\prime}}f_{\gamma^{\prime},(\gamma^{\prime})^g}})^{\varphi^{-1}}
$$
$$
=((\gamma^{\prime})^g,\delta^{h_{\gamma,\gamma^{\prime}}f_{\gamma^{\prime},(\gamma^{\prime})^g}h_{(\gamma')^g,\gamma}})
=((\gamma^{\prime})^g,\delta^{f^{\prime}}),
$$ 
where $f=h_{\gamma,\gamma'}f_{\gamma',(\gamma')^g}h_{(\gamma')^g,\gamma}$. In view of Eq.~\eqref{property}, we have $f'\in H_{\gamma,\gamma}$. Thus, $\varphi f\varphi^{-1}\in H_{\gamma,\gamma}\wr K$ and hence $\varphi W\varphi^{-1}\leq H_{\gamma,\gamma}\wr K$. The reverse inclusion is proved similarly. 
\end{proof}

\subsection{Algorithms}

Let $X$ be a colored digraph with vertex set $\Omega$ of cardinality~$n$
and $\tau$ a linear ordered set of subsets of~$\Omega$. Using the well-known
\emph{$2$-dimensional Weisfeiler-Leman algorithm}~\cite{Weis}, one can
construct in time~$\poly(n)$ a colored digraph
$X^{\prime}=\WL_2(X,\tau)$ such that every vertex-color class of $X$
and every element of~$\tau$ is a union of some vertex-color classes
of~$X^{\prime}$ and each arc-color class of $X$ is a union of some
arc-color classes of $X^{\prime}$. Moreover, if $\tau$ is
$\aut(X)$-invariant, then
\begin{equation}\label{automorphgraphs}
X^{\prime}=\WL_2(X,\tau)\succeq X,
\end{equation}
where the relation~$\succeq$ is as in Subsection~2.1.

In the following lemma, we collect some known results; the permutation groups of degree $n$ in the statement (as well throughout the rest of the paper) are given by generating sets of size~$\poly(n)$.

\begin{lemm}\label{solvable}
Let $K\leq \sym(\Omega)$ be a solvable group. Then the problems 
\begin{enumerate}
\item[{\rm (1)}] given a colored digraph~$X$, find the group $\aut(X)\cap K$,

\item[{\rm (2)}]  given a hypergraph~$\mathcal{H}$, find the group $\aut(\mathcal{H})\cap K$,

\item[{\rm (3)}]  given colored tournaments $T$  and $T'$, find the set $\iso(T,T^{\prime})$,
\end{enumerate}
can be solved in time~$\poly(n)$, $\poly(m)$, and $n^{O(\log n)}$, respectively, where $n$ is the vertex number of $X$ and $T$, and $m$ is the size of $\mathcal{H}$.
\end{lemm}

\begin{proof}
The algorithms for Problems~$(1)$ and~$(3)$ are well-known and can be
found in~\cite[Corollary~3.6]{BL} and in~\cite[Theorem~4.1]{BL},
respectively. For Problem~$(2)$, one can use Miller's
algorithm~\cite[Section~2]{Miller}, which runs in polynomial time if the group~$K$ is solvable (see~\cite{BCP}).
\end{proof}

\section{Auxiliary algorithm}

In this section, we construct an algorithm to be used for the proof of
Theorem~\ref{main}.  The general idea goes back to the Babai-Luks
procedure of finding a canonical labeling for a bipartite graph with
respect to a group action on one of the parts~\cite{BL}. However, the
algorithm below uses more information from the input graph which, in
turn, restricts the output group. Below, the restriction of a permutation group $G\le\sym(\Omega)$ 
to  a $G$-invariant subset $\Delta\subset \Omega$ is denoted by $G^\Delta$.\medskip

\begin{center}
\textbf{Algorithm AUX}
\end{center}
\medskip

\noindent\textbf{Input:} A colored digraph $X$ with two vertex-color classes
$\Gamma$, $\Delta$, and also  $X_{\Delta}$ is a tournament, a nonempty
intersection $D$ of some arc-color class with $\Gamma\times \Delta$, and
a group $K\geq \aut(X)^{\Gamma}$.\medskip

\noindent \textbf{Output:} Either a colored digraph $X^{\prime}\succ
X$, or a group $L\leq \sym(D)$  and a
homomorphism $\varphi:L\rightarrow \sym(\Delta)$ such that
$L^{\varphi}\geq \aut(X)^{\Delta}$.\medskip

\noindent\textbf{Step 1.} For every $\gamma\in \Gamma$, we use the
bijection $\psi:\gamma D\rightarrow \{\gamma\}\times\gamma D$, $\delta
\mapsto (\gamma,\delta)$, to construct the tournament
$X(\gamma)=(X_{\gamma D})^{\psi}$ on $\{\gamma\}\times \gamma
D$. \medskip

\noindent\textbf{Step 2.} For all $\gamma,\gamma^{\prime}\in \Gamma$,
find $H_{\gamma,\gamma^{\prime}}=\iso(X(\gamma),X(\gamma^{\prime}))$
by the algorithm from Lemma~\ref{solvable}(3). Construct the
partition~$\pi$ of $\Gamma$ such that $\gamma$ and $\gamma^{\prime}$
are in the same class of $\pi$ if and only if
$H_{\gamma,\gamma^{\prime}}\neq \varnothing$. \medskip
 
\noindent\textbf{Step 3.} If $|\pi|>1$, then output
$X^{\prime}=\WL_2(X,\pi)$. \medskip

\noindent\textbf{Step 4.} Construct the permutation group $W=W(X,K)$
on~$D$ by Eq.~\eqref{wreath} and the hypergraph $\mathcal{H}$ on $D$
with hyperedge set $E=\{\delta D^* \times \{\delta\}:~\delta\in
\Delta\}$, where $D^*=\{(\delta,\gamma):\ (\gamma,\delta)\in D\}$. \medskip

\noindent\textbf{Step 5.} Output the group $L=\aut(\mathcal{H})\cap W$
found by the algorithm from Lemma~\ref{solvable}(2), and the
homomorphism $\varphi:L\rightarrow \sym(\Delta)$ induced by the
bijection $E\rightarrow \Delta$, $\delta D^*\times \{\delta\}
\mapsto\delta$.\medskip

Denote by $n=|\Delta\cup\Gamma|$ the vertex number of~$X$, and by $k$ the maximal outdegree of the digraph with vertex set $\Delta\cup \Gamma$ and arc set~$D$.

\begin{prop}\label{aux}
Algorithm AUX correctly computes the stated output. Moreover, if $K$
is a solvable group, then $L$ is also solvable and the algorithm runs
in time $k^{O(\log k)}\poly(n)$.
\end{prop}

\begin{proof}
To prove the correctness assume first that the algorithm terminates at
Step~$3$. Obviously, $\pi$ is $\aut(X)$-invariant and hence
$X^{\prime}\succeq X$ by Eq.~\eqref{automorphgraphs}. Moreover,
$X^{\prime}\succ X$, because~$|\pi|>1$.

Now let the algorithm terminates at Step~$5$. Note that $\aut(X)$ acts
on $D$. We claim that
\begin{equation}\label{inclusion}
\aut(X)^D\leq W,
\end{equation}
where $W$ is the group found at Step~4. Indeed, let $a\in \aut(X)$. Since $K\geq \aut(X)^{\Gamma}$, we have $\gamma^a=\gamma^k$ for some $k\in K$ and all $\gamma\in \Gamma$. Now let $(\gamma,\delta)\in D$. Then $\delta^a\in (\gamma D)^a=\gamma^a D=\gamma^k D$. At Step~4, the set $H_{\gamma,\gamma^k}=\iso(X(\gamma),X(\gamma^k))$ is nonempty and hence $\delta^a=\delta^{f_{\gamma,\gamma^k}}$ for some $f_{\gamma,\gamma^k}\in H_{\gamma,\gamma^k}$. Therefore,
$$(\gamma,\delta)^a=(\gamma^k,\delta^{f_{\gamma,\gamma^k}}).$$
This is true for all $(\gamma,\delta)\in D$. Consequently, the permutation $a\in \sym(D)$ belongs to $W$ (see Lemma~\ref{wr}). Thus, Eq.~\eqref{inclusion} holds. 

By the definition of~$D$, the set $E$ defined at Step~4 is $\aut(X)^D$-invariant. Therefore, $\aut(X)^D\leq \aut(\mathcal{H})$, where $\mathcal{H}$ is the hypergraph defined at the same step. In view of Eq.~\eqref{inclusion}, this yields that 
$$\aut(X)^D\leq W\cap \aut(\mathcal{H})=L.$$ 
Thus, $\aut(X)^{\Delta}=(\aut(X)^D)^{\varphi}\leq L^{\varphi}$, which completes the proof of correctness. 

To estimate the running time of Algorithm AUX, we note that Steps~$1$,
$3$, and~$4$ can obviously be implemented in time~$\poly(n)$.  By
Lemma~\ref{solvable}(3), Step~2 can be done in time~$k^{O(\log
  k)}m^2+\poly(m)$, where $m=|\Gamma|$. Next, the group $W$ is
permutationally isomorphic to the wreath product~$\aut(X(\gamma)) \wr
K$ for every $\gamma\in \Gamma$ by Lemma~\ref{wr}. Note that the group
$\aut(X(\gamma))$ is solvable because $X(\gamma)$ is a
tournament. Therefore, if $K$ is solvable, so are $W$ and $L\leq W$.
Besides, the hypergraph~$\mathcal{H}$ has $|D|=mk$ vertices,
$|\Delta|\leq mk$ hyperedges, and each hyperedge consists of at
most~$m$ vertices. Therefore, the size of $\mathcal{H}$ is polynomial
in~$n$. Thus, the group $L=\aut(\mathcal{H})\cap W$ (and hence the
homomorphism $\varphi$) can be found in time~$\poly(n)$ by
Lemma~\ref{solvable}(2) and, consequently, Step~5 terminates within
the same time bound.
\end{proof}

\section{The proof of Theorem~\ref{main}}
 
We divide the proof of Theorem~\ref{main} into two parts. First, we prove Theorem~\ref{main2} in which a fixed-parameter tractable algorithm for finding the automorphism group of a $k$-spanning colored tournament is constructed. Then we derive Theorem~\ref{main} from Theorem~\ref{main2}.

\begin{theo}\label{main2}
Let $k\geq 1$, $X$ a $k$-spanning colored tournament with~$n$ vertices, and $m$ the minimal size of a vertex-color class of $X$. Then one can find the group $\aut(X)$ in time~$(m^{O(\log m)}+k^{O(\log k)})\poly(n)$.
\end{theo}

\begin{proof}
In what follows, $\Omega$ is the vertex set of~$X$.
It suffices to construct a solvable group $K\geq \aut(X)$ in the required time: indeed, then the group $\aut(X)$ can be found inside $K$ in time $\poly(n)$ by Lemma~\ref{solvable}(1). To describe the algorithm constructing~$K$, we need an auxiliary lemma.

\begin{lemm}\label{classes}
Let $\Sigma$ be a proper union of vertex-color classes of~$X$. Then there exist vertex-color classes $\Gamma$ inside $\Sigma$ and $\Delta$ outside $\Sigma$, and an arc-color class of $X$ of maximal valency at most~$k$ whose intersection~$D$ with $\Gamma\times \Delta$ is nonempty.
\end{lemm}

\begin{proof}
Since $X$ is $k$-spanning, there exist $\alpha\in \Sigma$ and $\beta\in \Omega\setminus \Sigma$ such that the arc $(\alpha,\beta)$ belongs to some color class~$s$ of maximal valency at most~$k$. Let $\Gamma$ and $\Delta$ be the vertex-color classes of $X$ containing $\alpha$ and $\beta$, respectively. Then $\Gamma\subseteq \Sigma$ and $\Delta\subseteq \Omega\setminus \Sigma$. This completes the proof with $D=s\cap (\Gamma\times \Delta)$.
\end{proof}

The algorithm starts with finding a vertex-color class $\Sigma$ of $X$ of size~$m$, and the group $K=\aut(X_{\Sigma})$ by the algorithm from Lemma~\ref{solvable}(3). Since $X$ and hence $X_{\Sigma}$ is a tournament, the group $K$ is solvable and the cost of this step is essentially $m^{O(\log m)}$. Next, the algorithm proceeds as follows:\medskip

\quad\textbf{while} $\Sigma\neq \Omega$ \textbf{do}\smallskip

\begin{verse}

\noindent find classes $\Gamma$, $\Delta$ and a set $D$ as in Lemma~\ref{classes}; apply the algorithm AUX to $Y=X_{\Gamma\cup\Delta}$, $\Gamma,\Delta,D$, and the group $K^{\Gamma}$;\medskip

\noindent \textbf{if} the output of AUX is a digraph $Y'$, \medskip

\noindent \textbf{then} change $X$ by the tournament obtained by replacing the vertex-color classes $\Gamma$ and $\Delta$ with those of $Y'$;  \textbf{break}; \medskip

\noindent // now the output of AUX is $L\leq \sym(D)$ and $\varphi:L\rightarrow \sym(\Delta)$\medskip

\noindent \textbf{else} $\Sigma:=\Sigma \cup \Delta$ and $K:=K\times L^{\varphi}$; 
\end{verse}

\quad\textbf{od} \medskip

\noindent If the above loop terminates with break, then we repeat the whole procedure with the new colored tournament~$X$. Since obviously $X^{\prime}\succ X$, such a repetition occurs at most~$n^2$ times. Moreover, at each iteration of the loop we have $K\geq \aut(X)^{\Sigma}$: at the zero step this is clear, and then this follows from the inclusion $L^{\varphi}\geq \aut(X)^{\Delta}$. Thus, at the end, 
$$K\geq \aut(X).$$
Further, at each iteration of the loop the group $K$ is solvable by induction and the fact that $L^{\varphi}$ is solvable (see Proposition~\ref{aux}). Therefore, the cost of each repetition is at most  $m^{O(\log m)}+k^{O(\log k)}\poly(n)$.
\end{proof}

It is well known that the problems of finding the set of all
isomorphisms and the automorphism group are polynomial-time equivalent
in the class of all graphs. However, the standard proof of this
equivalence does not work for $k$-spanning colored
tournaments. Lemma~\ref{reduction}, proved below and based on an old
observation of Babai, is sufficient to deduce Theorem~\ref{main} from
Theorem~\ref{main2} for $m=1$.

\begin{lemm}\label{reduction}
Given $k$-spanning colored tournaments $X$ and $Y$ with~$n$ vertices,
one can efficiently construct a set $\mathcal{T}=\mathcal{T}(X,Y)$ of
$n$ colored tournaments such that
\begin{enumerate}
\item[{\rm(a)}] each tournament from $\mathcal{T}$ has $3n+1$ vertices, is
  $k'$-spanning with $k'=\max \{3,k\}$, and has a singleton vertex-color
  class,
\item[{\rm(b)}]  the set $\iso(X,Y)$ can efficiently be found inside the set
  $\bigcup \limits_{T\in \mathcal{T}} \aut(T)$.
\end{enumerate}
\end{lemm}

\begin{proof} 
Let us fix a vertex $\alpha$ of $X$. For every vertex $\beta$ of
  $Y$ we define a tournament~$T=T_{\beta}$ as the disjoint union of
  $X$, $Y$, $X^{\prime}=X$ with attached new vertex~$\mu$, and the sets
\begin{equation}\label{colors}
(\Omega \times \Delta)\cup (\Delta \times \Omega^{\prime})\cup (\Omega^{\prime}\times \Omega)~\text{and}~\{\mu\}\times (\Omega \cup \Delta \cup \Omega^{\prime})
\end{equation}
of new arcs, where $\Omega$, $\Delta$, and $\Omega^{\prime}$ are the
vertex sets of $X$, $Y$, and $X^{\prime}$, respectively. The colors of
vertices in $\Omega \cup \Delta \cup \Omega^{\prime}$ are the
same as in $X$, $Y$, and $X^{\prime}$, and $\mu$ gets a new
color. Thus, $T$ contains $3n+1$ vertices and the singleton
vertex-color class~$\{\mu\}$.

The colors of arcs inside $X$, $Y$, and $X^{\prime}$ are not changed,
the arcs from the first set from Eq.~\eqref{colors} get a new color,
and the arcs in the second set are divided into two classes one of
which consists of the arcs $(\mu,\alpha)$, $(\mu,\beta)$, and
$(\mu,\alpha')$, where $\alpha'$ is the copy of~$\alpha$ in~$X'$. 
Since the union of the latter class with the
arc-color classes of $X$ and $Y$ of valency at most~$k$ is obviously a
connected relation, the colored tournament~$T$ is $k'$-spanning. Thus,
the set $\mathcal{T}=\{T_{\beta}:~\beta\in \Delta\}$ satisfies
condition~$(a)$. \medskip

\textbf{Claim.} $X$ and $Y$ are isomorphic if and only if there exist
$\beta\in \Delta$ and $f\in\aut(T_{\beta})$ such that
$\alpha^f=\beta$, $\beta^f=\alpha^{\prime}$, and
$(\alpha^{\prime})^f=\alpha$.

\begin{proof}
To prove the ``if'' part, it suffices to show that
$\Omega^f=\Delta$. Assume that $\gamma^f\notin \Delta$ for some
$\gamma\in \Omega$. Since $Y$ is a tournament, $(\alpha,\gamma)$ or $(\gamma,\alpha)$ is an arc of $T_{\beta}$; for definiteness, assume the first. Then the arcs
$(\alpha,\gamma)$ and $(\alpha^f,\gamma^f)=(\beta,\gamma^f)$ are in
different arc-color classes of $T_{\beta}$. It follows that
$f\notin\aut(T_{\beta})$, a contradiction. Conversely, let $f_0\in
\iso(X,Y)$ and $\beta=\alpha^{f_0}$. Then the permutation $f$ on the
union $\Omega \cup \Delta\cup \Omega^{\prime}\cup \{\mu\}$, defined by
$$\mu^f=\mu,\qquad f^{\Omega}=f_0, \qquad f^{\Delta}=f_0^{-1}, \qquad f^{\Omega^{\prime}}=\id,$$ 
is obviously an automorphism of $T_{\beta}$.
\end{proof}

To complete the proof, for any $\beta\in \Delta$ denote by $H_{\beta}$
the set consisting of all permutations of $\aut(T_{\beta})$ taking
$\alpha$, $\beta$, $\alpha^{\prime}$ to $\beta$, $\alpha^{\prime}$,
$\alpha$, respectively (note that if $H_{\beta}\neq \varnothing$, then
$H_{\beta}$ is a coset of the point stabilizer
$\aut(T_{\beta})_{\alpha,\beta,\alpha^{\prime}}$). Then by the Claim,
we have
$$\iso(X,Y)=\bigcup \limits_{\beta\in\Delta} H_{\beta}.$$
To prove that the set $\mathcal{T}$ satisfies condition~$(b)$, it remains to note that the index of $\aut(T_{\beta})_{\alpha,\beta,\alpha^{\prime}}$ in $\aut(T_{\beta})$ is at most~$n^3$, and hence $H_{\beta}$ can be found in time~$\poly(n)$ for every $\beta\in \Delta$. 
\end{proof}

\end{document}